\documentclass[a4paper]{amsart}
\usepackage{graphicx}
\usepackage{amssymb}
\usepackage{amsmath}
\usepackage{amsthm}
\usepackage{amscd}
\usepackage[all,2cell]{xy}

\UseAllTwocells \SilentMatrices
\newtheorem{thm}{Theorem}[section]

\newtheorem{cor}[thm]{Corollary}
\newtheorem{lem}[thm]{Lemma}

\newtheorem{prop}[thm]{Proposition}
\theoremstyle{definition}

\theoremstyle{remark}
\newtheorem{rem}[thm]{\bf Remark}
\numberwithin{equation}{section}

\begin{document}
\title[A note on separable functors and monads]
{A note on separable functors and monads}

\author[Xiao-Wu Chen] {Xiao-Wu Chen}

\thanks{The author is supported by National Natural Science Foundation of China (No. 11201446) and NCET-12-0507.}
\date{\today}

\thanks{E-mail:
xwchen$\symbol{64}$mail.ustc.edu.cn}

\maketitle

\dedicatory{}%
\commby{}%

\begin{abstract}
For an adjoint pair $(F, G)$ of functors, we prove that $G$ is a separable functor if and only if the defined monad
is separable  and the  associated  comparison functor is an equivalence up to retracts. In this case, under an idempotent
completeness condition,  the adjoint pair $(F, G)$ is monadic. This applies to the comparison between the
derived category of the category of equivariant objects in an abelian category and  the category
of equivariant objects in the derived category of the abelian category.
\end{abstract}

\section{Introduction}

Let $\mathcal{A}$ be a category and  $G$ be a finite group. By a strict action of $G$ on $\mathcal{A}$, we mean a group
homomorphism from $G$ to the automorphism group of $\mathcal{A}$. Then we form the category $\mathcal{A}^G$ of $G$-equivariant objects
in $\mathcal{A}$; compare \cite{RR, DGNO}.

We assume that $\mathcal{A}$ is an abelian category, and thus the category $\mathcal{A}^G$ is also abelian. Consider the bounded
derived category $\mathbf{D}^b(\mathcal{A})$. Then the $G$-action on $\mathcal{A}$ extends to $\mathbf{D}^b(\mathcal{A})$.
In general, the categories $\mathbf{D}^b(\mathcal{A}^G)$ and $\mathbf{D}^b(\mathcal{A})^G$ are not equivalent. However,
a nice observation in \cite[Lemma 1.1]{Po} claims that they are equivalent under a characteristic zero condition and
a hereditary condition on the abelian category $\mathcal{A}$. One might ask whether these conditions for this equivalence are
essential or not. It turns out that this equivalence holds in a great generality. Indeed, separable functors and
monads appear naturally in the construction of the category of equivariant objects.  All these motivates this note.

The note is organized as follows. In Section 2, we collect some basic facts on separable functors. In Section 3, we recall some
facts on separable monads and the construction of the associated comparison functor to an adjoint pair. Then  we prove
that in an adjoint pair $(F, G)$ of functors, the functor $G$ is separable if and only if the defined monad is separable
and the associated comparison functor is an equivalence up
to retracts; see Proposition \ref{prop:main}. In this case, if we assume an idempotent completeness condition,
then the adjoint pair $(F, G)$
is monadic; see Corollary \ref{cor:sepfm}. In Section 4, we apply these results to obtain two triangle equivalences; in particular,
we prove that if the order of the group $G$ is invertible in an abelian category $\mathcal{A}$, then there is a triangle equivalence
between  $\mathbf{D}^b(\mathcal{A}^G)$ and $\mathbf{D}^b(\mathcal{A})^G$; see Proposition \ref{prop:app2}.  Here, we mention
that there exists a canonical (pre-)triangulated structure on $\mathbf{D}^b(\mathcal{A})^G$ by applying the results in \cite{Bal}.

\section{Separable functors}

In this section, we recall  from \cite{NVO,Rae} some basic facts on separable functors.

Let $\mathcal{C}$ be a category. We will consider the Hom
 bifunctor ${\rm Hom}_\mathcal{C}(-, -)\colon \mathcal{C}^{\rm op}\times \mathcal{C}\rightarrow \mathbf{\rm Set}$;
 here, $\mathcal{C}^{\rm op}$ denotes the opposite category of $\mathcal{C}$, and $\mathbf{\rm Set}$ denotes the category of sets.

Let $F\colon \mathcal{C}\rightarrow \mathcal{D}$ be a functor. Then we have the bifunctor
 ${\rm Hom}_\mathcal{D}(F-, F-)\colon  \mathcal{C}^{\rm op}\times \mathcal{C}\rightarrow \mathbf{\rm Set}$; moreover,
 we have a natural transformation induced the action of $F$ on morphisms
\begin{align*}
F\colon {\rm Hom}_\mathcal{C}(-, -) \longrightarrow {\rm Hom}_\mathcal{D}(F-, F-).
\end{align*}

The functor $F\colon \mathcal{C}\rightarrow \mathcal{D}$ is \emph{separable} \cite{NVO} provided that the above
natural transformation $F$ admits a retraction $H$. In other words, for each pair of objects $X, Y$ in $\mathcal{C}$ there exists a map
$$H_{X, Y}\colon  {\rm Hom}_\mathcal{D}(F(X), F(Y))\longrightarrow {\rm Hom}_\mathcal{C}(X, Y)$$
satisfying that $H_{X, Y}(F(f))=f$ for any morphism $f\colon X\rightarrow Y$; moreover, $H$ is functorial in both $X$ and $Y$. It follows
that a separable functor is faithful. On the other hand, a fully faithful functor is separable.

\begin{lem}\label{lem:sepa}
Let $F\colon \mathcal{C}\rightarrow \mathcal{D}$ and $G\colon \mathcal{D}\rightarrow \mathcal{E}$ be two functors.  Then the following statements hold.
\begin{enumerate}
\item If both $F$ and $G$ are separable, then the composite $GF$ is separable.
\item If the composite $GF$ is separable, then $F$ is separable.
\item If $F'\colon \mathcal{C}\rightarrow \mathcal{D}$ is a separable functor and there exist natural
transformations $\phi\colon F'\rightarrow F$ and $\psi\colon F\rightarrow F'$ satisfying
$\psi \circ \phi={\rm Id}_{F'}$, then the functor $F$ is separable;
\item Assume that $\mathcal{C}=\mathcal{D}$ and that there exist natural transformations $\phi\colon{\rm Id}_\mathcal{C}\rightarrow F$
and $\psi\colon F \rightarrow {\rm Id}_\mathcal{C}$ satisfying $\psi\circ \phi={\rm Id}$. Then the functor $F$ is separable.
\end{enumerate}
\end{lem}

\begin{proof}
We refer to \cite[Lemma 1.1]{NVO} for (1) and (2). The statement (4) is a special case of (3), since the identity functor is
always separable.

For (3), consider the natural transformation $$\Delta\colon {\rm Hom}_\mathcal{D}(F-, F-)\longrightarrow {\rm Hom}_\mathcal{D}(F'-, F'-)$$
given by $\Delta_{X, Y}(g)=\psi_Y\circ g\circ \phi_X$ for any morphism $g\colon F(X)\rightarrow F(Y)$. Then for any morphism
 $f\colon X\rightarrow Y$ in $\mathcal{C}$, we apply the identities $\psi_Y\circ F(f)=F'(f)\circ \psi_X$ and
 $\psi\circ \phi={\rm Id}_{F'}$ to deduce  $\Delta_{X,Y}(F(f))=F'(f)$; in other words, we have an identity $\Delta\circ F=F'$ of
  natural transformations. Since the natural transformation $F'$ has a retraction, so does $F$.
\end{proof}

Assume that $F\colon \mathcal{C}\rightarrow \mathcal{D}$ admits a right adjoint $G\colon \mathcal{D}\rightarrow \mathcal{C}$.
We denote by $\eta\colon {\rm Id}_\mathcal{C}\rightarrow GF$ the unit and $\epsilon\colon FG\rightarrow {\rm Id}_\mathcal{D}$ the
counit; they satisfy $\epsilon F\circ F\eta ={\rm Id}_F$ and $G\epsilon\circ \eta G={\rm Id}_G$. In what follows, by referring
to an adjoint pair $(F, G)$ we really mean the quadruple $(F, G; \eta, \epsilon)$.

The following result is due to \cite[Theorem 1.2]{Rae}. We make slight modification and include a short proof.

\begin{lem}\label{lem:adjoint}
Let $(F, G)$ be an adjoint pair as above. Then the following statements are equivalent:
\begin{enumerate}
\item the functor $G$ is separable;
\item there exists a natural transformation $\xi\colon {\rm Id}_\mathcal{D}\rightarrow FG$ satisfying $\epsilon\circ \xi={\rm Id}$;
\item there exist natural transformations $\phi\colon {\rm Id}_\mathcal{D} \rightarrow FG$ and $\psi\colon FG\rightarrow {\rm Id}_\mathcal{D}$
satisfying $\psi\circ \phi={\rm Id}$.
    \end{enumerate}
\end{lem}

\begin{proof}
The implication ``$(2)\Rightarrow (3)$" is trivial, and ``$(3)\Rightarrow (1)$" follows from Lemma \ref{lem:sepa}(4) and (2).

 It remains to prove ``$(1)\Rightarrow (2)$". For this, we identify the bifunctor ${\rm Hom}_\mathcal{C}(G-, G-)$ with
 ${\rm Hom}_\mathcal{D}(FG-, -)$ via the adjoint pair $(F, G)$. Then the natural transformation $G\colon {\rm Hom}_\mathcal{D}(-, -)
 \rightarrow {\rm Hom}_\mathcal{C}(G-, G-)$  is identified with the following natural transformation
$$G'\colon {\rm Hom}_\mathcal{D}(-, -) \longrightarrow {\rm Hom}_\mathcal{D}(FG-, -),
$$
where $G'_{X, Y}(f)=f\circ \epsilon_X$ for any morphism $f\colon X\rightarrow Y$. Then $G$ admits a retraction
if and only if so does $G'$. By Yoneda Lemma for contravariant functors,  a natural transformation
$H'_{X, Y}\colon {\rm Hom}_\mathcal{D}(FG(X), Y)\rightarrow {\rm Hom}_\mathcal{D}(X, Y)$ is uniquely induced by
 morphisms $\xi_X\colon X\rightarrow FG(X)$, that is, it sends $g$ to $g\circ \xi_X$; moreover, $\xi$ is
 natural in $X$. Then $H'\circ G'={\rm Id}$ implies that $\epsilon \circ \xi={\rm Id}$.
\end{proof}

We mention that Lemma \ref{lem:adjoint}(3) implies that the setting of \cite[3.2]{RR} really deals with an adjoint pair
consisting of two separable functors with extra properties.

\section{Separable monads}

In this section, we recall basic facts on monads and modules. We characterize separable functors using
 separable monads and the associated comparison functor.

Let $\mathcal{C}$ be a category. Recall from  \cite[Chapter VI]{McL} that a \emph{monad} on  $\mathcal{C}$
is a triple $(M, \eta, \mu)$ consisting of an endofunctor $M\colon \mathcal{C}\rightarrow \mathcal{C}$ and two natural transformations, the \emph{unit} $\eta\colon{\rm Id}_\mathcal{C}\rightarrow M$  and the \emph{multiplication} $\mu \colon M^2\rightarrow M$, subject to the relations $\mu \circ M\mu =\mu \circ \mu M$
and $\mu\circ M\eta={\rm Id}_M=\mu\circ \eta M$. We sometimes denote the monad by $M$ when $\eta$ and $\mu$ are understood.

A monad $(M, \eta, \mu)$ is \emph{separable} provided that there exists a natural
transformation $\sigma\colon M\rightarrow M^2$ satisfying that $\mu\circ \sigma={\rm Id}_M$ and
$M\mu\circ \sigma M=\sigma \circ \mu=\mu M\circ M\sigma$; see \cite[Section 6]{BV}.

One associates to each adjoint pair $(F, G; \eta, \epsilon)$ on two
categories $\mathcal{C}$ and $\mathcal{D}$ a monad $(GF, \eta, \mu)$ on $\mathcal{C}$,
where $\mu=G\epsilon F\colon M^2=GFGF\rightarrow G{\rm Id}_\mathcal{D}F=M$. The monad
$(GF, \eta, \mu)$ is said to be \emph{defined} by the adjoint pair $(F, G)$

We observe the following fact, which relates separable functors to separable monads.

\begin{lem}\label{lem:sepmon}
Consider the adjoint pair $(F, G; \eta, \epsilon)$. If the functor $G$ is separable, then the defined
monad $(GF, \eta, \mu)$ is separable.
\end{lem}

\begin{proof}
Since $G$ is separable, by Lemma \ref{lem:adjoint}(2) there exists a natural transformation $\xi\colon {\rm Id}_\mathcal{D}\rightarrow FG$ with the property $\epsilon \circ \xi={\rm Id}$. Set $\sigma=G\xi F\colon M=G{\rm Id}_\mathcal{D}F\rightarrow GFGF=M^2$. Then we have $\mu\circ \sigma={\rm Id}_M$.
The remaining identity follows from the identity $FG\epsilon \circ \xi FG=\xi\circ \epsilon=\epsilon FG\circ FG\xi$,
while these two equalities follow from the naturalness of $\xi$ and $\epsilon$, respectively.
\end{proof}

 In what follows, we take the notation from \cite{Bal}. For a monad $M$, an $M$-\emph{module} is a
 pair $(X, \lambda)$ consisting of an object $X$ in $\mathcal{C}$ and a morphism $\lambda \colon M(X)\rightarrow X$
 subject to the conditions $\lambda\circ M\lambda =\lambda\circ \mu_X$ and $\lambda\circ \eta_X={\rm Id}_X$;
  the object $X$ is said to be the underlying object of the module. A morphism  $f\colon (X, \lambda)\rightarrow (X', \lambda')$
  of two $M$-modules is a morphism $f\colon  X\rightarrow X'$ in $\mathcal{C}$ satisfying $f\circ \lambda=\lambda'\circ M(f)$.
  This gives rise to the category  $M\mbox{-Mod}_\mathcal{C}$ of $M$-modules. For each object $X$ in $\mathcal{C}$, we have
   the corresponding $M$-module $(M(X), \mu_X)$, the \emph{free module}.

For each monad $(M, \eta, \mu)$ on $\mathcal{C}$, there is a classical construction of an adjoint pair that
defines the given monad; see \cite[IV.2]{McL}. Consider the functor $F_M\colon \mathcal{C}\rightarrow M\mbox{-Mod}_\mathcal{C}$
sending $X$ to the free module $(M(X), \mu_X)$, and a morphism $f\colon X\rightarrow Y$ to a morphism
$M(f)\colon (M(X), \mu_X)\rightarrow (M(Y), \mu_Y)$. Denote by $G_M\colon M\mbox{-Mod}_\mathcal{C}\rightarrow \mathcal{C}$
the forgetful functor. Then we have the adjoint pair $(F_M, G_M; \eta, \epsilon_M)$, where for an
$M$-module $(X, \lambda)$, $(\epsilon_M)_{(X, \lambda)}=\lambda$. Here, we use $M=G_M F_M$.
This adjoint pair $(F_M, G_M; \eta, \epsilon_M)$ defines the given monad $M$.

The following result is due to \cite[2.9(1)]{BBW}; compare \cite[Proposition 6.3]{BV}.

\begin{lem}\label{lem:sepmon2}
Let $M$ be a monad on $\mathcal{C}$, and let $(F_M, G_M)$ be the adjoint pair as above. Then $G_M$ is a separable functor if and only if $M$ is a separable monad.
\end{lem}

\begin{proof}
The ``only if" part follows from Lemma \ref{lem:sepmon}, since $(F_M, G_M)$ defines $M$. For the ``if" part, assume that there is a natural transformation $\sigma\colon M\rightarrow M^2$ subject to the required conditions. We define a natural transformation $\xi\colon {\rm Id}_{M\mbox{-}{\rm Mod}_\mathcal{C}}\rightarrow F_M G_M$ as follows: for any $M$-module $(X, \lambda)$, $\xi_{(X, \lambda)}=M(\lambda)\circ \sigma_X\circ \eta_X$; here, we use that $F_M G_M(X, \lambda)=(M(X), \mu_X)$. Then $\epsilon_M\circ \xi={\rm Id}$. By Lemma \ref{lem:adjoint}, the functor $G_M$ is separable.
\end{proof}

The above adjoint pair $(F_M, G_M; \eta, \epsilon_M)$ enjoys a universal property: for any adjoint pair $(F, G; \eta, \epsilon)$ on $\mathcal{C}$ and $\mathcal{D}$ that defines $M$, there is a unique functor $K\colon \mathcal{D}\rightarrow M\mbox{-Mod}_\mathcal{C}$ satisfying $KF=F_M$ and $G_MK=G$; see \cite[IV.3]{McL}. This unique functor $K$ will be referred as the \emph{comparison functor} associated to the  adjoint pair $(F, G; \eta, \epsilon)$;  it is given by $K(D)=(G(D), G\epsilon_D)$ for any object $D$ and $K(f)=G(f)$ for any morphism $f$.

Following \cite[IV.3]{McL} the adjoint pair $(F, G)$ is \emph{monadic} (resp. \emph{strictly monadic}) if the associated
comparison functor $K\colon \mathcal{D}\rightarrow M\mbox{-Mod}$ is an equivalence (resp. an isomorphism). In these cases, we might identify $\mathcal{D}$ with $M\mbox{-Mod}_\mathcal{C}$.

The following fact is well known, which is implicit in \cite[IV.3 and IV.5]{McL}.

\begin{lem}\label{lem:ff}
Consider the comparison functor $K\colon \mathcal{D}\rightarrow M\mbox{-{\rm Mod}}_\mathcal{C}$ associated to the adjoint pair $(F, G)$. Then $K$ is fully faithful on ${\rm Im}\; F$.
\end{lem}

Here, for any functor $F\colon \mathcal{C}\rightarrow \mathcal{D}$ we denote by ${\rm Im}\; F$ the \emph{image} of $F$, that is, the full subcategory of $\mathcal{D}$ consisting of objects of the form $F(X)$ for objects $X$ in $\mathcal{C}$.

\begin{proof}
Recall that $KF=F_M$. Then for any objects $X$ and $Y$ in $\mathcal{C}$, $K$ induces a map ${\rm Hom}_\mathcal{D}(F(X), F(Y))\rightarrow {\rm Hom}_{M\mbox{-}{\rm Mod}_\mathcal{C}}(F_M(X), F_M(Y))$. But by the adjunctions, both the Hom sets are identified to ${\rm Hom}_\mathcal{C}(X, M(Y))$. Using these identifications, the induced map becomes the identity; here, we use that $G_MK=G$. In other words, the induced map is bijective. Then we are done.
\end{proof}

The final ingredient we need is the idempotent completion of a category. Let  $\mathcal{C}$ be a category. An idempotent morphism $e\colon X\rightarrow X$ \emph{splits} if there exist two morphisms $u\colon X\rightarrow Y$ and $v\colon Y\rightarrow X$ satisfying $e=v\circ u$ and ${\rm Id}_Y=u\circ v$; in this case, $Y$ is said to be a \emph{retract} of $X$. If all idempotents split, the category $\mathcal{C}$ is said to be \emph{idempotent complete}.

There is a well-known construction of the \emph{idempotent completion} $\mathcal{C}^\natural$ of
a category $\mathcal{C}$. The category $\mathcal{C}^\natural$ is defined as follows: the objects
are pairs $(X, e)$, where $X$ is an object in $\mathcal{C}$ and $e\colon X\rightarrow X$ is an idempotent;
a morphism $f\colon (X, e)\rightarrow (X', e')$ is a morphism $f\colon X\rightarrow X'$ in $\mathcal{C}$
satisfying $f=e'\circ f\circ e$. The canonical functor $\iota_\mathcal{C}\colon \mathcal{C}\rightarrow \mathcal{C}^\natural$,
sending $X$ to $(X, {\rm Id}_X)$, is fully faithful; it is an equivalence if and only if $\mathcal{C}$ is idempotent complete.

Any functor $F\colon \mathcal{C}\rightarrow \mathcal{D}$ extends to a functor $F^\natural\colon \mathcal{C}^\natural\rightarrow \mathcal{D}^\natural$ by means of $F^\natural (X, e)=(F(X), F(e))$ and $F^\natural (f)=F(f)$. We have $\iota_\mathcal{D} F=F^\natural \iota_\mathcal{C}$. The functor $F\colon \mathcal{C}\rightarrow \mathcal{D}$ is called an \emph{equivalence up to retracts} provided that $F^\natural$ is an equivalence.

The following facts are direct to verify.

\begin{lem}\label{lem:idem}
Let $F\colon \mathcal{C}\rightarrow \mathcal{D}$ be a functor, and let $\mathcal{C}'\subseteq \mathcal{C}$ be
a full subcategory such that each object of $\mathcal{C}$ is a retract of some object in $\mathcal{C}'$.
Then the following two statements hold:
\begin{enumerate}
\item the functor $F$ is fully faithful if and only if so is its restriction to $\mathcal{C}'$;
  \item the functor $F$ is an equivalence up to retracts if and only if $F$ is fully faithful and each
  object $Y$ in $\mathcal{D}$ is a retract of an object in ${\rm Im}\; F$;
  \item if $\mathcal{C}$ is idempotent complete, then $F$ is an equivalence if and only if it is an equivalence up to retracts. \hfill $\square$
      \end{enumerate}
\end{lem}

The main result  of this section is as follows, where separable functors are characterized using separable
 monads and the associated comparison functor. It slightly extends  Lemma \ref{lem:sepmon2}. We mention that the
 result, at least in the triangulated case, is implicit in \cite[Theorem 5.17(d)]{Bal}.

\begin{prop}\label{prop:main}
Let $(F, G)$ be an adjoint pair on categories $\mathcal{C}$ and $\mathcal{D}$. Consider the defined monad $M=GF$ on $\mathcal{C}$ and the associated comparison functor $K\colon \mathcal{D}\rightarrow M\mbox{-{\rm Mod}}_\mathcal{C}$. Then the functor $G$ is separable if and only if $M$ is a separable monad and $K\colon \mathcal{D}\rightarrow M\mbox{-{\rm Mod}}_\mathcal{C}$ is an equivalence up to retracts.
\end{prop}

\begin{proof}
For the ``only if" part, we know already by Lemma \ref{lem:sepmon} that $M$ is a separable monad; moreover, then by Lemma \ref{lem:sepmon2} the functor $G_M$ is separable.

 Since $G$ is separable, there exists $\xi\colon {\rm Id}_\mathcal{D}\rightarrow FG$ such that $\epsilon\circ\xi={\rm Id}$. Then  any object $X$ in $\mathcal{D}$ is a retract of $FG(X)$, in particular,  an object from ${\rm Im}\; F$. By Lemma \ref{lem:ff} the restriction of $K$ to ${\rm Im}\; F$ is fully faithful. Then Lemma \ref{lem:idem}(1) implies that $K$ is fully faithful. Similarly, the separability of $G_M$ implies that each $M$-module is a retract of a module in ${\rm Im}\; F_M\subseteq {\rm Im}\; K$. Then Lemma \ref{lem:idem}(2) implies that $K$ is an equivalence up to retracts.

 For the ``if" part, we observe that $G_M$ is separable by Lemma \ref{lem:sepmon2}, and $K$ is separable, since it is fully faithful. Hence, by Lemma \ref{lem:sepa}(1) the composite $G=G_MK$ is separable.
\end{proof}

We observe the following immediate consequence of Proposition \ref{prop:main} and Lemma \ref{lem:idem}(3). In particular, 
under an idempotent completeness condition,
an adjoint pair $(F, G)$ with $G$ separable is always monadic.

  \begin{cor}\label{cor:sepfm}
  Keep the notation as above. Assume further that $\mathcal{D}$ is idempotent complete. Then the functor $G$ is separable if and only if  $M$ is a separable monad and $K\colon \mathcal{D}\rightarrow M\mbox{-{\rm Mod}}_\mathcal{C}$ is an equivalence. \hfill $\square$
  \end{cor}

\section{Applications to triangle equivalences}

In this section, we apply the equivalence in Corollary \ref{cor:sepfm} to obtain two triangle equivalences. In particular, we obtain
a comparison result between the derived category of the category of equivariant objects and the category of equivariant objects in
the derived category; see Proposition \ref{prop:app2}.

Let $\mathcal{A}$ be an abelian category. A monad $(M, \eta, \mu)$ on $\mathcal{A}$ is \emph{exact}
if the endofunctor $M\colon \mathcal{A}\rightarrow \mathcal{A}$ is exact, in particular, it is additive.
In this case, the category $M\mbox{-Mod}_\mathcal{A}$ of $M$-modules is abelian; indeed, a
sequence of $M$-modules is exact if and only if the sequence of the underlying objects is exact.
It follows that both functors $F_M\colon \mathcal{A}\rightarrow M\mbox{-Mod}_\mathcal{A}$ and
$G_M\colon M\mbox{-Mod}_\mathcal{A}\rightarrow \mathcal{A}$ are exact.

We consider the bounded derived category $\mathbf{D}^b(\mathcal{A})$. The exact monad  $(M, \eta, \mu)$
extends to a monad on the derived category, that is, the monad acts on complexes componentwise.
The resulting monad on $\mathbf{D}^b(\mathcal{A})$ is still denoted by $M$.  Since both the functors
 $F_M$ and $G_M$ are exact,  they extends to triangle functors
 $\mathbf{D}^b(F_M)\colon \mathbf{D}^b(\mathcal{A})\rightarrow \mathbf{D}^b(M\mbox{-Mod}_\mathcal{A})$
 and $\mathbf{D}^b(G_M)\colon \mathbf{D}^b(M\mbox{-Mod}_\mathcal{A})\rightarrow \mathbf{D}^b(\mathcal{A})$.
 They still form an adjoint pair, which defines the monad $M$ on $\mathbf{D}^b(\mathcal{A})$. Therefore,
 we have the following comparison functor associated to the adjoint pair $(\mathbf{D}^b(F_M), \mathbf{D}^b(G_M))$
\begin{align*}
K\colon \mathbf{D}^b(M\mbox{-Mod}_\mathcal{A})\longrightarrow M\mbox{-Mod}_{\mathbf{D}^b(\mathcal{A})}.
\end{align*}
The functor $K$ sends a complex  $\cdots \rightarrow (X^n, \lambda^n)\rightarrow (X^{n+1},
\lambda^{n+1})\rightarrow\cdots $ to $(X^\bullet, \lambda^\bullet)$, a module of the monad  $M$ on $\mathbf{D}^b(\mathcal{A})$.

Assume that the given exact monad $M$ on $\mathcal{A}$ is separable. Then the corresponding
monad $M$ on $\mathbf{D}^b(\mathcal{A})$ is also separable. Indeed, the section $\sigma$ for $\mu$
extends to the corresponding section on $\mathbf{D}^b(\mathcal{A})$ componentwise; moreover, it follows
that the monad $M$ on $\mathbf{D}^b(\mathcal{A})$ is stably separable in the sense of \cite[Definition 3.5]{Bal}.
 Recall that the bounded derived category $\mathbf{D}^b(\mathcal{A})$ is idempotent complete; see \cite[Corollary 2.10]{BS}.
 Then we apply \cite[Corollary 4.3]{Bal} to obtain that the category $M\mbox{-Mod}_{\mathbf{D}^b(\mathcal{A})}$
  carries a canonical pre-triangulated structure, that is, a triangulated structure possibly without
  the octahedral axiom; indeed, a triangle in  $M\mbox{-Mod}_{\mathbf{D}^b(\mathcal{A})}$ is exact if and
  only if the corresponding triangle of the underlying objects in $\mathbf{D}^b(\mathcal{A})$ is exact.
   It follows that the comparison functor $K$ is a triangle functor.

We obtain the following triangle equivalence for separable exact monads, which is analogous to \cite[Theorem 6.5]{Bal}.

\begin{prop}\label{prop:app1}
Let $M$ be a monad on an abelian category $\mathcal{A}$, which is exact and separable.
Then the comparison functor $K\colon \mathbf{D}^b(M\mbox{-{\rm Mod}}_\mathcal{A})\rightarrow M\mbox{-{\rm Mod}}_{\mathbf{D}^b(\mathcal{A})}$  is a triangle equivalence.
\end{prop}

\begin{proof}
We already proved that $K$ is a triangle functor. By Lemmas \ref{lem:sepmon2} and \ref{lem:adjoint},
the functor $G_M\colon M\mbox{-Mod}_\mathcal{A}\rightarrow \mathcal{A}$ is separable and thus the
counit $\epsilon_M\colon F_MG_M \rightarrow {\rm Id}_{M\mbox{-}{\rm Mod}_\mathcal{A}}$ has a section.
The counit of the adjoin pair $(\mathbf{D}^b(F_M), \mathbf{D}^b(G_M)))$ is induced by $\epsilon_M$,
and thus also has a section. Then Lemma \ref{lem:adjoint} yields that the functor
$\mathbf{D}^b(G_M)\colon \mathbf{D}^b(M\mbox{-Mod}_\mathcal{A})\rightarrow \mathbf{D}^b(\mathcal{A})$ is
separable. By \cite[Corollary 2.10]{BS} the bounded derived category  $\mathbf{D}^b(M\mbox{-Mod}_\mathcal{A})$ is
idempotent complete. Then it follows from Corollary \ref{cor:sepfm} that the comparison functor $K$ is an equivalence.
\end{proof}

\begin{rem}
We mention that if the unbounded derived categories involved are both idempotent complete, then
the same result holds for unbounded derived categories.
\end{rem}

In what follows, we apply Proposition \ref{prop:app1} to obtain a more concrete triangle equivalence.

 We assume temporarily that $\mathcal{A}$ is an arbitrary category. Let $G$ be a finite g
 roup, which is written multiplicatively and whose unit is denoted by $e$. We assume that there is a \emph{strict}
 action of $G$ on $\mathcal{A}$, that is, there is a group homomorphism from $G$ to the automorphism
 group of $\mathcal{A}$. For $g\in G$ and a morphism $\theta\colon X\rightarrow Y$ in $\mathcal{A}$,
 the action by $g$ is denoted by $^g\theta\colon {^g X}\rightarrow {^g Y}$.  A \emph{$G$-equivariant object}
 in $\mathcal{A}$ is a pair $(X, \alpha)$, where $X$ is an object in $\mathcal{A}$ and $\alpha$ assigns for each
 $g\in G$ an isomorphism $\alpha_g\colon X\rightarrow {^gX}$ subject to the relations
 $^g(\alpha_{g'})\circ \alpha_g=\alpha_{gg'}$. A morphism $\theta\colon (X, \alpha)\rightarrow (Y, \beta)$ of
  two $G$-equivariant objects is a morphism $\theta\colon X\rightarrow Y$ such that
  $\beta_g\circ \theta={^g\theta}\circ \alpha_g$ for all $g\in G$. This gives rise to
   the category $\mathcal{A}^G$ of $G$-equivariant objects, and the forgetful functor
   $U\colon \mathcal{A}^G\rightarrow \mathcal{A}$ defined by $U(X, \alpha)=X$. For details,
   we refer to \cite{RR,DGNO}.

 Let $\mathcal{A}$ be an additive category. Then the forgetful functor $U$
 admits a left adjoint $F\colon \mathcal{A}\rightarrow \mathcal{A}^G$ which is defined as follows:
  for an object $X$, set $F(X)=(\oplus_{h\in G} {^hX}, {\rm Id})$, where
  ${\rm Id}_g\colon \oplus_{h\in G} {^hX}\rightarrow {^g(\oplus_{h\in G} {^hX})}$ is the identity map for
  any $g\in G$; the functor $F$ sends a morphism $\theta\colon X\rightarrow Y$ to $\oplus_{h\in G} {^h\theta}$.
  For an object $X$ in $\mathcal{A}$ and an object $(Y, \beta)$ in $\mathcal{A}^G$,
  a morphism $F(X)\rightarrow (Y, \beta)$ is of the form $\sum_{h\in G}\theta_h\colon \oplus_{h\in G} {^h X}\rightarrow Y$
  satisfying $^g(\theta_h)=\beta_g\circ \theta_{gh}$ for any $g, h\in G$.  The adjunction of $(F, U)$ is given by the
  following natural isomorphism
\begin{align*}
 {\rm Hom}_{\mathcal{A}^G} (F(X), (Y, \beta))\longrightarrow {\rm Hom}_\mathcal{A}(X, U(Y, \beta))
 \end{align*}
 sending $\sum_{h\in G} \theta_h$ to $\theta_e$. The corresponding unit $\eta\colon {\rm Id}_\mathcal{A}\rightarrow UF$ is
 given such that $\eta_X=({\rm Id}_X, 0, \cdots, 0)^t$, where `t' denotes the transpose;
 the counit $\epsilon\colon FU\rightarrow {\rm Id}_{\mathcal{A}^G}$ is given such that
  $\epsilon_{(Y, \beta)}=\sum_{h\in G} \beta_h^{-1}$.

The following fact seems  well known.

 \begin{lem}\label{lem:monadic}
 Let $\mathcal{A}$ be an additive category and $G$ be a finite group acting on $\mathcal{A}$ strictly.
 Then the adjoint pair $(F,U; \eta, \epsilon)$ is strictly monadic.
 \end{lem}

\begin{proof}
We compute the defined monad $(M=UF, \eta, \mu)$ of the adjoint pair $(F, U)$.
Then $M(X)=\oplus_{h\in G} {^h X}$ and $M(\theta)=\oplus_{h\in G} {^h \theta}$ for a morphism $\theta$ in $\mathcal{A}$.
 The multiplication $\mu$ is given by
 $$\mu_X=U\epsilon_{F(X)}\colon M^2(X)=\oplus_{h, g\in G} {^{hg}X}\rightarrow M(X)=\oplus_{h\in G} {^h X}$$
 with the property that the corresponding component
  $^{hg}X\rightarrow {^{h'}X}$ is $\delta_{hg, h'}{\rm Id}_{(^{h'}X)}$; here, $\delta$ is the Kronecker symbol.

  An $M$-module is a pair $(X, \lambda)$ with $\lambda=\sum_{h\in G}\lambda_h\colon M(X)=\oplus_{h\in G} {^h X}\rightarrow X$.
  The condition $\eta_X\circ \lambda={\rm Id}_X$ is equivalent to $\lambda_e={\rm Id}_X$,
  and $\lambda\circ M(\lambda)=\lambda\circ \mu_X$ is equivalent to
   $\lambda_{hg}=\lambda_h\circ {^h (\lambda_g)}$ for any $h, g\in G$. Hence, if we
   set $\alpha_h\colon X\rightarrow {^hX}$ to be $(\lambda_h)^{-1}$, we obtain an object $(X, \alpha)\in \mathcal{A}^G$.
    Roughly speaking, the morphism $\lambda$ carries the same information with $\alpha$.

  Indeed, the associated comparison functor $K\colon \mathcal{A}^G\rightarrow M\mbox{-Mod}_\mathcal{A}$ sends $(X, \alpha)$ to $(X, \lambda)$ by $\lambda_h=(\alpha_h)^{-1}$. It follows immediately that $K$ induces a bijection on objects, and is fully faithful, thus an isomorphism of categories.
\end{proof}

Let $\mathcal{A}$ be an additive category. A natural number $n$ is said to be \emph{invertible}
 in $\mathcal{A}$ provided that for any morphism $f\colon X\rightarrow Y$ there exists a unique
 morphism $g\colon X\rightarrow Y$ such that $f=ng$. This unique morphism is denoted by $\frac{1}{n}f$.
 For example, if $\mathcal{A}$ is a $k$-linear for a field $k$ and the characteristic
 of $k$ does not divide $n$, then $n$ is invertible in $\mathcal{A}$; see \cite[p.255]{RR}.

The third statement of the following result is an application of the results in \cite{Bal}.

\begin{lem}\label{lem:G-equ}
 Let $\mathcal{A}$ be an additive category and $G$ be a finite group acting on $\mathcal{A}$ strictly.
 Assume that the order $|G|$ of $G$ is invertible in $\mathcal{A}$. Then the following statements hold.
  \begin{enumerate}
  \item The forgetful functor $U\colon \mathcal{A}^G\rightarrow \mathcal{A}$ is separable;
  \item The monad $M=UF$ on $\mathcal{A}$ is separable;
  \item Assume that $\mathcal{A}$ is a pre-triangulated category which is idempotent complete, and
  that the action of $G$ on $\mathcal{A}$ is given by triangle automorphisms. Then $\mathcal{A}^G$ has a
  unique pre-triangulated structure such that the forgetful functor $U$ is a triangle functor.
  \end{enumerate}
\end{lem}

\begin{proof}
For (1), we apply Lemma \ref{lem:adjoint}, and thus it suffices to
prove that the counit $\epsilon\colon FU\rightarrow {\rm Id}_{\mathcal{A}^G}$ admits a section $\xi$.
We define a natural transformation $\xi\colon {\rm Id}_{\mathcal{A}^G}\rightarrow FU$ such that
 $\xi_{(X, \alpha)}=\frac{1}{|G|}\prod_{h\in G}\alpha_h\colon (X, \alpha)\rightarrow (\oplus_{h\in G} {^h X}, {\rm Id})$.
 It follows that $\epsilon\circ\xi={\rm Id}$. The statement (2) follows from Lemma \ref{lem:sepmon}.

For (3), we identify by Lemma \ref{lem:monadic} that category $\mathcal{A}^G$ with $M\mbox{-Mod}_\mathcal{A}$.
The monad $M$ is a triangle functor and by (2) it is separable; it is indeed stably separable in the sense of
\cite[Definition 3.5]{Bal} by the explicit construction of the section $\xi$ above. Then the statement
follows from \cite[Corollary 4.3]{Bal}.
\end{proof}

Let $\mathcal{A}$ be an abelian category. Assume that there is a strict $G$-action on $\mathcal{A}$. Then
the category $\mathcal{A}^G$ is abelian and the functors $F\colon \mathcal{A}\rightarrow \mathcal{A}^G$ and
 $U\colon \mathcal{A}^G\rightarrow \mathcal{A}$ are both exact. We will consider the bounded derived category $\mathbf{D}^b(\mathcal{A}^G)$. Then the strict action of $G$ on $\mathcal{A}$ extends to $\mathbf{D}^b(\mathcal{A})$.
We have the following  functor
$$K\colon \mathbf{D}^b(\mathcal{A}^G)\longrightarrow \mathbf{D}^b(\mathcal{A})^G,$$
sending a complex $\cdots \rightarrow (X^n, \alpha^n)\rightarrow (X^{n+1}, \alpha^{n+1})\rightarrow \cdots$ in
$\mathcal{A}^G$ to a $G$-equivariant object $(X^\bullet, \alpha^\bullet)$ in $\mathbf{D}^b(\mathcal{A})$.

Assume that $|G|$ is invertible in $\mathcal{A}$, and thus $|G|$ is
invertible in $\mathbf{D}^b(\mathcal{A})$. Then by Lemma \ref{lem:G-equ}(3) the
category $\mathbf{D}^b(\mathcal{A})^G$ has a unique pre-triangulated structure such that the forgetful
functor $U\colon \mathbf{D}^b(\mathcal{A})^G\rightarrow \mathbf{D}^b(\mathcal{A})$ is a triangle functor.
It follows that the above functor $K$ is a triangle functor.

The following comparison  result extends the nice observation in \cite[Lemma 1.1]{Po}, where extra conditions are put
for the triangle equivalence.

\begin{prop}\label{prop:app2}
Let $\mathcal{A}$ be an abelian category and $G$ be a finite group acting on $\mathcal{A}$ strictly.
Assume that the order $|G|$ of $G$ is invertible in $\mathcal{A}$. Then the above
functor $K\colon \mathbf{D}^b(\mathcal{A}^G)\rightarrow \mathbf{D}^b(\mathcal{A})^G$ is a triangle equivalence.
\end{prop}

\begin{proof}
 We consider the monad $M=UF$ on $\mathcal{A}$;
it is exact and separable. It extends
to a monad $M$ on $\mathbf{D}^b(\mathcal{A})$. By Lemma \ref{lem:monadic} we
identify $\mathcal{A}^G$ with $M\mbox{-Mod}_\mathcal{A}$, $\mathbf{D}^b(\mathcal{A})^G$
with $M\mbox{-Mod}_{\mathbf{D}^b(\mathcal{A})}$. Then the triangle equivalence
follows from Proposition \ref{prop:app1}.
\end{proof}

\begin{rem}
(1) Proposition \ref{prop:app2} might still hold if the action is not strict. Here, we recall that a
(non-strict) action of a finite group $G$ on $\mathcal{A}$ is a monoidal functor from $\underline{G}$ to
the category of endomorphism functors of $\mathcal{A}$; here, $\underline{G}$ denotes the corresponding
monoidal category of $G$. For details, we refer to \cite[Section 4]{DGNO}.  We need to adapt Lemmas
\ref{lem:monadic} and \ref{lem:G-equ} for non-strict actions; compare \cite[Theorem 1.4]{RR}.

(2) The assumption on the invertibility of $|G|$ in Proposition \ref{prop:app2} is necessary. Indeed, if $\mathcal{A}=k\mbox{-Mod}$ is
the category of $k$-modules, where $k$ is a field such that its characteristic divides $|G|$. Take the trivial action of $G$
on $\mathcal{A}$, and thus $\mathcal{A}^G$ is isomorphic to the category $kG\mbox{-Mod}$ of modules over the group
algebra $kG$. The functor  $K\colon \mathbf{D}^b(\mathcal{A}^G)\rightarrow \mathbf{D}^b(\mathcal{A})^G$ is not an equivalence;
indeed,  $\mathbf{D}^b(\mathcal{A}^G)$ is a triangulated category with non-split triangles, but $\mathbf{D}^b(\mathcal{A})^G$
is an abelian category with non-split extensions.
\end{rem}

\vskip 10pt

\noindent {\bf Acknowledgements}\quad The author thanks Jianmin Chen for the reference \cite{Po} and
Guodong Zhou for the reference \cite{Bal}, and thanks Zengqiang Lin for pointing out several misprints.

\bibliography{}

\vskip 10pt

 {\footnotesize \noindent Xiao-Wu Chen \\
 School of Mathematical Sciences,
  University of Science and Technology of
China, Hefei 230026, Anhui, PR China \\
Wu Wen-Tsun Key Laboratory of Mathematics, USTC, Chinese Academy of Sciences, Hefei 230026, Anhui, PR China.\\
URL: http://home.ustc.edu.cn/$^\sim$xwchen.}

\end{document}